\documentclass{amsart}
\usepackage{setspace}
\usepackage{a4}
\usepackage{amsthm}
\usepackage{latexsym}
\usepackage{amsfonts}
\usepackage{graphicx}
\usepackage{textcomp}
\usepackage{cite}
\usepackage{enumerate}
\usepackage{amssymb}
\usepackage{hyperref}
\usepackage{amsmath}
\usepackage{tikz}
\usepackage[mathscr]{euscript}
\usepackage{mathtools}
\newtheorem{theorem}{Theorem}[section]

\newtheorem{definition}[theorem]{Definition}
\newtheorem{example}[theorem]{Example}

\newtheorem{problem}[theorem]{Problem}
\newtheorem{proposition}[theorem]{Proposition}

\title{This is the title}
\usepackage{fancyhdr}

\begin{document}
\hrule\hrule\hrule\hrule\hrule
\vspace{0.3cm}	
\begin{center}
{\bf{Functional Delsarte-Goethals-Seidel-Kabatianskii-Levenshtein-Pfender Bound}}\\
\vspace{0.3cm}
\hrule\hrule\hrule\hrule\hrule
\vspace{0.3cm}
\textbf{K. Mahesh Krishna}\\
School of Mathematics and Natural Sciences\\
Chanakya University Global Campus\\
NH-648, Haraluru Village\\
Devanahalli Taluk, 	Bengaluru  Rural District\\
Karnataka State, 562 110, India\\
 Email: kmaheshak@gmail.com

Date: \today
\end{center}

\hrule\hrule
\hrule\hrule 
\vspace{0.5cm}
\textbf{Abstract}: Pfender \textit{[J. Combin. Theory Ser. A, 2007]} provided a one-line proof for a variant of the Delsarte-Goethals-Seidel-Kabatianskii-Levenshtein upper bound for spherical codes, which offers an upper bound for the celebrated (Newton-Gregory) kissing number problem. Motivated by this proof, we introduce the notion of codes in pointed metric spaces (in particular on Banach spaces) and derive a nonlinear (functional) Delsarte-Goethals-Seidel-Kabatianskii-Levenshtein-Pfender upper bound for spherical codes. We also introduce nonlinear (functional) Kissing Number Problem.



\hrule

\hrule
\section{Introduction}
A finite set of points satisfying certain conditions on the unit sphere in a Euclidean space plays an important role in many areas of Mathematics, Physics, Chemistry, Engineering, and Biology \cite{SAFFKUIJLAARS}. It is also one of the oldest concepts still being studied in science. Following are a few famous problems studying finite points on unit sphere.
\begin{enumerate}[\upshape(i)]
	\item Tammes problem \cite{CONWAYSLOANE}.
	\item Thomson problem (Problem 7 in the \textbf{Mathematical problems for the next century} by Smale) \cite{SMALE}.
	\item Spherical designs problem \cite{DELSARTEGOETHALSSEIDEL}.
	\item Spherical codes problem \cite{DELSARTEGOETHALSSEIDEL}.
	\item Equiangular lines problem \cite{LEMMENSSEIDEL}.
\end{enumerate}

In this paper, we are concerned about spherical codes. We recall the definition. 
\begin{definition}\cite{ZONG}
 Let $d\in \mathbb{N}$ and $\theta \in [0, 2\pi)$. A set $\{\tau_j\}_{j=1}^n$  of unit vectors  in $\mathbb{R}^d$	is said to be \textbf{$(d,n,\theta )$-spherical code} in $\mathbb{R}^d$ if 
\begin{align*}
	\langle \tau_j, \tau_k\rangle\leq \cos \theta , \quad \forall 1\leq j, k \leq n, j \neq k.
\end{align*}		
\end{definition}
 Fundamental problem associated with spherical codes is the following.
 \begin{problem}\label{SCP}
 	Given $d$ and $\theta$, what is the maximum $n$ such that there exists a $(d,n,\theta )$-spherical code $\{\tau_j\}_{j=1}^n$  in $\mathbb{R}^d$?
 \end{problem}
 The case $\theta=\pi/3$ is the famous \textbf{(Newton-Gregory) kissing number problem}. With extensive efforts from many mathematicians, it is  resolved only in following cases.
 \begin{enumerate}[\upshape(i)]
 	\item $d=1$, $n=2$ \cite{PFENDERZIEGLER}.
 	\item $d=2$, $n=6$ \cite{PFENDERZIEGLER}.
 	\item $d=3$, $n=12$ (Schutte and van der Waerden in 1953 \cite{SCHUTTEVANDER}). 
 	\item $d=4$, $n=24$ (due to Musin in 2008 \cite{MUSIN}).
 	\item  $d=8$, $n=240$ (due to Odlyzko and Sloane  in  1979 \cite{ODLYZKOSLOANE}).
 	\item 	 $d=24$, $n=196560$ (due to Odlyzko and Sloane in 1979 \cite{ODLYZKOSLOANE}).
 \end{enumerate}
 As spherical codes generalize kissing number problem, much less is known for spherical codes. We refer \cite{DELSARTEGOETHALSSEIDEL, ERICSONZINOVIEV, CONWAYSLOANE, COHNJIAOKUMARTORQUATO, BANNAIBANNAI, BACHOCVALLENTIN, JENSSENJOOSPERKINS, SARDARIZARGAR} for more information on spherical codes.
 Problem \ref{SCP} has connection even with sphere packing  \cite{COHNZHAO}. 
 
 So far, the most used method for obtaining upper bounds on spherical codes is the Delsarte-Goethals-Seidel-Kabatianskii-Levenshtein
bound which we recall. Let $n \in \mathbb{N}$ be fixed. The Gegenbauer polynomials are defined inductively as 
\begin{align*}
	G_0^{(n)}(r)&\coloneqq1,\quad \forall r \in [-1,1],\\
	G_1^{(n)}(r)&\coloneqq r,\quad \forall r \in [-1,1],\\
	&\quad\vdots\\
	G_k^{(n)}(r)&\coloneqq\frac{(2k+n-4)rG_{k-1}^{(n)}(r)-(k-1)	G_{k-2}^{(n)}(r)}{k+n-3}, \quad \forall r \in [-1,1], \quad \forall k \geq 2.
\end{align*}
Then the family $\{G_k^{(n)}\}_{k=0}^\infty$ is orthogonal on the interval $[-1,1] $ with respect to the weight 
\begin{align*}
	\rho(r)\coloneqq (1-r^2)^\frac{n-3}{2}, \quad \forall r \in [-1,1].
\end{align*}
\begin{theorem} \cite{DELSARTEGOETHALSSEIDEL, ERICSONZINOVIEV} (\textbf{Delsarte-Goethals-Seidel-Kabatianskii-Levenshtein Linear Programming Bound}) \label{DGS}
	Let $\{\tau_j\}_{j=1}^n$  be a  $(d,n,\theta )$-spherical code in $\mathbb{R}^d$. 
	Let $P$ be a real polynomial satisfying following conditions.
	\begin{enumerate}[\upshape(i)]
		\item $P(r)\leq 0$ for all $-1\leq r\leq \cos \theta$.
		\item Coefficients in the  Gegenbauer expansion 
		\begin{align*}
			P=\sum_{k=0}^{m}a_kG_k^{(n)}
		\end{align*}
	satisfy 
	\begin{align*}
a_0>0, \quad a_k\geq 0, ~\forall 1\leq k \leq m.
	\end{align*}
	\end{enumerate}
Then
\begin{align*}
	n \leq \frac{P(1)}{a_0}.
\end{align*}
\end{theorem}
Theorem \ref{DGS} is primarily used to obtain upper bounds for spherical codes in any given dimension. Lower bounds are then derived using specific constructions of codes in each dimension. If we find a code with a specific number of points in a given dimension that also matches the upper bound from Theorem \ref{DGS}, we obtain an optimal code. However, finding both tight upper bounds and explicit code constructions in each dimension is extremely challenging due to the mysterious behavior of codes across dimensions.

Motivated from Theorem \ref{DGS},   Pfender in 2007 derived following breakthrough result.
 \begin{theorem}  \cite{PFENDER} (\textbf{Delsarte-Goethals-Seidel-Kabatianskii-Levenshtein-Pfender Bound}) \label{PFENDERT}
 	Let $\{\tau_j\}_{j=1}^n$  be a  $(d,n,\theta )$-spherical code in $\mathbb{R}^d$. Let $c>0$ and $\phi:[-1,1]\to \mathbb{R}$ be a function satisfying following.
 	\begin{enumerate}[\upshape(i)]
 		\item 
 		\begin{align*}
 			\sum_{j=1}^n\sum_{k =1}^n\phi(\langle \tau_j, \tau_k\rangle)\geq 0.	
 		\end{align*}
 		\item $\phi(r)+c\leq 0$ for all $-1\leq r \leq \cos \theta$.
 	\end{enumerate}
 Then 
 \begin{align*}
 	n\leq \frac{\phi(1)+c}{c}.
 \end{align*}
In particular, if $\phi(1)+c\leq 1$, then $n\leq 1/c$.
 \end{theorem}
The beauty of Theorem \ref{PFENDERT} is that it does not require polynomials; any function satisfying the conditions of Theorem \ref{PFENDERT} is sufficient. Using this theorem, Pfender obtained improved upper bounds for the kissing number in dimensions 10, 16, 17, 25, and 26. He also derived new bounds for spherical codes in dimensions 3, 4, and 5. Furthermore, Pfender provided new proofs for the kissing numbers in dimensions 3 and 4.\\

In \cite{POLYA}, Polya gave the advice: ‘vary the problem’ whenever the solution to a problem is not known. In the context of Banach spaces, the following are a few such generalizations to metric spaces.
\begin{enumerate}[\upshape(i)]
	\item Metric type (due to Bourgain, Milman and Wolfson in 1986 \cite{BOURGAINMILMANWOLFSON}).
	\item Metric cotype (due to Mendel and Naor in 2008 \cite{MENDELNAOR}).
	\item Lipschitz p-summing operators (due to Farmer and Johnson in 2009 \cite{FARMERJOHNSON}).
	\item Metric frames (due to Krishna and Johnson in 2022 \cite{KRISHNAJOHNSON}).
	\item Regular simplex in a metric space (due to Cohn, Kumar and Minton in 2016 \cite{COHNKUMARMINTON}).
	\item Metric packings (due to Viazovska in 2018 \cite{VIAZOVSKA}).
\end{enumerate}
We carefully  formulate the notion of codes in pointed metric spaces (Definition \ref{MC}) and show that it matches with the existing spherical codes whenever we have Euclidean space (Proposition \ref{R}). We show that bound of Delsarte-Goethals-Seidel-Kabatianskii-Levenshtein-Pfender  can be extended for pointed metric spaces (in particular, for Banach spaces).

 \section{Metric Codes}
  Let $(\mathcal{M}, 0)$ be a pointed metric space. The collection 	$\operatorname{Lip}_0(\mathcal{M}, \mathbb{R})$
 is defined as $\operatorname{Lip}_0(\mathcal{M},\mathbb{R})\coloneqq \{f:\mathcal{M} 
 \rightarrow \mathbb{R}  \operatorname{ is ~ Lipschitz ~ and } f(0)=0\}.$
 For $f \in \operatorname{Lip}_0(\mathcal{M}, \mathbb{R})$, the Lipschitz norm
 is defined as 
 \begin{align*}
 	\|f\|_{\operatorname{Lip}_0}\coloneqq \sup_{x, y \in \mathcal{M}, x\neq
 		y} \frac{|f(x)-f(y)|}{d(x,y)}.
 \end{align*}
 We introduce metric  codes as follows.
 \begin{definition}\label{MC}
Let $(\mathcal{M}, 0)$ be a pointed metric space with metric $m$. 	 For $1\leq j \leq n$, let $f_j\in \operatorname{Lip}_0(\mathcal{M}, \mathbb{R})$ and $\tau_j \in \mathcal{M}$. The pair $ (\{f_j \}_{j=1}^n, \{\tau_j \}_{j=1}^n) $ is said to be a  \textbf{$(n,\theta)$-metric code} or \textbf{$(n,\theta)$-nonlinear code} or   \textbf{$(n,\theta)$-Lipschitz  code} in $ \mathcal{M}$
if following conditions hold.
	\begin{enumerate}[\upshape(i)]
	\item $\|f_j\|_{\operatorname{Lip}_0}=1$ for all $1\leq j \leq n$.
		\item $m(\tau_j,0)= 1$ for all $1\leq j \leq n$.
	\item 	$f_j(\tau_j)= 1$ for all $1\leq j \leq n$.
	\item $f_j(\tau_k)\leq \cos \theta $ for all $1\leq j,k \leq n, j \neq k$.
	\end{enumerate}
We call the case $\theta=\pi/3$  as the \textbf{nonlinear kissing number problem}.
 \end{definition}
For Banach spaces, we define (linear) functional codes as follows.
 \begin{definition}\label{BK}
	Let $\mathcal{X}$ be a real Banach space. 	 For $1\leq j \leq n$, let $f_j\in \mathcal{X}^*$ and $\tau_j \in \mathcal{X}$. The pair $ (\{f_j \}_{j=1}^n, \{\tau_j \}_{j=1}^n) $ is said to be a  \textbf{$(n,\theta)$-functional code} in $ \mathcal{X}$
	if following conditions hold.
	\begin{enumerate}[\upshape(i)]
		\item $\|f_j\|=1$ for all $1\leq j \leq n$.
		\item $\|\tau_j\|= 1$ for all $1\leq j \leq n$.
		\item 	$f_j(\tau_j)= 1$ for all $1\leq j \leq n$.
		\item $f_j(\tau_k)\leq \cos \theta $ for all $1\leq j,k \leq n, j \neq k$.
	\end{enumerate}
We call the case $\theta=\pi/3$  as the \textbf{functional kissing number problem}.
\begin{example}
		Let $\mathcal{X}$ be any  real $d$ dimensional Banach space. Auerbach lemma \cite{WOJTASZCZYK} says that there exists a collection $ \{f_j \}_{j=1}^d\subseteq \mathcal{X}^*$ and a collection $ \{\tau_j \}_{j=1}^d \subseteq \mathcal{X}$ satisfying following conditions.
			\begin{enumerate}[\upshape(i)]
			\item $\|f_j\|=1$ for all $1\leq j \leq d$.
			\item $\|\tau_j\|= 1$ for all $1\leq j \leq d$.
			\item 	$f_j(\tau_j)= 1$ for all $1\leq j \leq d$.
			\item $f_j(\tau_k)=0 $ for all $1\leq j,k \leq d, j \neq k$.
		\end{enumerate}
	Then  $ (\{f_j \}_{j=1}^d, \{\tau_j \}_{j=1}^d) $ is a $(d,\theta)$-functional code in $ \mathcal{X}$, for any $\theta$. 
\end{example}
\begin{example}
	Let $\mathcal{H}$ be any  real $d$ dimensional Hilbert space. Let $n\geq d$ and  $ \{\tau_j \}_{j=1}^n \subseteq \mathcal{H}$ be a collection  in  $\mathcal{H}$  satisfying following conditions.
	\begin{enumerate}[\upshape(i)]
		\item $\|\tau_j\|= 1$ for all $1\leq j \leq n$.
		\item There is a $\gamma\geq 0$ such that $|\langle \tau_j, \tau_k\rangle|=\gamma $ for all $1\leq j,k \leq n, j \neq k$.
	\end{enumerate}
Collections $ \{\tau_j \}_{j=1}^n $ satisfying above two conditions are known as equiangular lines \cite{WALDRON}. Let  $\theta\geq 0$ be such that $\gamma \leq \cos \theta$. Then $\{\tau_j \}_{j=1}^n$  is a $(d,\theta)$-code in $ \mathcal{H}$. In particular, an orthonormal basis is a $(d,\pi/2)$-code in $ \mathcal{H}$.
\end{example}
\begin{example}
	Let $\mathcal{H}$ be any  real $d$ dimensional Hilbert space. Let   $ \{\tau_j \}_{j=1}^n \subseteq \mathcal{H}$ be a collection of unit vectors  in  $\mathcal{H}$.  In particular, any finite normalized tight frames for Hilbert spaces \cite{BENEDETTOFICKUS}. Let $\theta\geq 0$ be such that 
	\begin{align*}
		\max_{1\leq j, k\leq n, j \neq k}\langle \tau_j, \tau_k \rangle \leq \cos \theta.
	\end{align*}
Then $\{\tau_j \}_{j=1}^n$  is a $(n,\theta)$-code in $ \mathcal{H}$.
\end{example}
\begin{example}
	Let $\mathcal{X}$ be any  real $d$ dimensional Banach space.  For $n \in \mathbb{N}$, let   $ \{f_j \}_{j=1}^n\subseteq \mathcal{X}^*$ and $ \{\tau_j \}_{j=1}^n \subseteq \mathcal{X}$ satisfy following conditions.
\begin{enumerate}[\upshape(i)]
	\item $\|f_j\|=1$ for all $1\leq j \leq n$.
	\item $\|\tau_j\|= 1$ for all $1\leq j \leq n$.
\end{enumerate}
In particular, any finite normalized frames for Banach spaces will satisfy above two conditions \cite{KORNELSON}. Let $\theta\geq 0$ be such that 
\begin{align*}
	\max_{1\leq j, k\leq n, j \neq k}f_j(\tau_k) \leq \cos \theta.
\end{align*}
Then  $ (\{f_j \}_{j=1}^n, \{\tau_j \}_{j=1}^n) $ is a $(n,\theta)$-functional code in $ \mathcal{X}$. 	
\end{example}
\begin{proposition}\label{R}
	For the space $\mathbb{R}^d$, Definition \ref{BK} matches with the spherical codes (in particular, with the kissing-number problem).
\end{proposition}
\begin{proof}
Let $ (\{f_j \}_{j=1}^n, \{\tau_j \}_{j=1}^n) $ be a $(n,\theta)$-functional code in $\mathbb{R}^d$. We need to show that $f_j$ is determined by $\tau_j$ for all $x\in \mathbb{R}^d$ and  for all $1\leq j \leq n$. Let $1\leq j \leq n$. From Riesz representation theorem, there exists a unique $w_j \in \mathbb{R}^d$ such that $f_j(x)=\langle x, \omega_j\rangle$  for all $x\in \mathbb{R}^d$ and $\|f_j\|=\|\omega_j\|$. Now we need to show that $\omega_j=\tau_j$. Since $\|f_j\|=1$, we must have $\|\omega_j\|=1$. But then
\begin{align*}
	1=f_j(\tau_j)=\langle \tau_j, \omega_j\rangle \leq \|\tau_j\|\|\omega_j\|=1.
\end{align*}
Therefore $\omega_j=\alpha \tau_j$  for some $\alpha \in \mathbb{R}$. The conditions $\|\omega_j\|=\|\tau_j\|=1$ and $\langle \tau_j, \omega_j\rangle =1$ then force $\omega_j=\tau_j$. 
\end{proof}
\end{definition}
Following is a nonlinear generalization of Theorem \ref{PFENDERT}.
\begin{theorem} (\textbf{Functional Delsarte-Goethals-Seidel-Kabatianskii-Levenshtein-Pfender Bound}) \label{FDGSP}
	Let $ (\{f_j \}_{j=1}^n, \{\tau_j \}_{j=1}^n) $  be a  $(n,\theta )$-metric  code in a  pointed metric space $\mathcal{M}$. Let $c>0$ and $\phi:[-1,1]\to \mathbb{R}$ be a function satisfying following.
\begin{enumerate}[\upshape(i)]
	\item 
	\begin{align*}
		\sum_{j=1}^n\sum_{k =1}^n\phi(f_j(\tau_k))\geq 0.	
	\end{align*}
	\item $\phi(r)+c\leq 0$ for all $-1\leq r \leq \cos \theta$.
\end{enumerate}
Then 
\begin{align*}
	n\leq \frac{\phi(1)+c}{c}.
\end{align*}
In particular, if $\phi(1)+c\leq 1$, then $n\leq 1/c$.	
\end{theorem}
\begin{proof}
	Define $\psi:[-1,1]\ni r \mapsto \psi(r)\coloneqq \phi(r)+c\in \mathbb{R}$. Then 
	\begin{align*}
	\sum_{j=1}^n\sum_{k =1}^n\psi(f_j(\tau_k))&=\sum_{j=1}^{n}\psi(f_j(\tau_j))+\sum_{1\leq j,k \leq n, j \neq k}\psi(f_j(\tau_k))\\
	&=\sum_{j=1}^{n}\psi(1)+\sum_{1\leq j,k \leq n, j \neq k}\psi(f_j(\tau_k))\\
	&=n(\phi(1)+c)+\sum_{1\leq j,k \leq n, j \neq k}(\phi(f_j(\tau_k))+c)\\
	&\leq n(\phi(1)+c)+0=n(\phi(1)+c).
	\end{align*}
We also have 
\begin{align*}
	\sum_{j=1}^n\sum_{k =1}^n\psi(f_j(\tau_k))=	\sum_{j=1}^n\sum_{k =1}^n(\phi(f_j(\tau_k))+c)=	\sum_{j=1}^n\sum_{k =1}^n\phi(f_j(\tau_k))+cn^2.
\end{align*}
Therefore 
\begin{align*}
	cn^2\leq 	\sum_{j=1}^n\sum_{k =1}^n\phi(f_j(\tau_k))+cn^2\leq n(\phi(1)+c).
\end{align*}
\end{proof}
Following generalization of Theorem \ref{FDGSP} is easy.
\begin{theorem}
	Let $ (\{f_j \}_{j=1}^n, \{\tau_j \}_{j=1}^n) $  be a  $(n,\theta )$-metric  code in a  pointed metric space $\mathcal{M}$. Let $c>0$ and 
	\begin{align*}
		\phi:\{f_j(\tau_k):1\leq j, k \leq n\}\to \mathbb{R}	
	\end{align*}
	be a function satisfying following.
	\begin{enumerate}[\upshape(i)]
		\item 
		\begin{align*}
			\sum_{j=1}^n\sum_{k =1}^n\phi(f_j(\tau_k))\geq 0.	
		\end{align*}
		\item $\phi(r)+c\leq 0$ for all $ r \in \{f_j(\tau_k):1\leq j, k \leq n, j \neq k\}$.
	\end{enumerate}
	Then 
	\begin{align*}
		n\leq \frac{\phi(1)+c}{c}.
	\end{align*}
	In particular, if $\phi(1)+c\leq 1$, then $n\leq 1/c$.	
\end{theorem}

 \bibliographystyle{plain}
 \bibliography{reference.bib}

\begin{thebibliography}{10}

\bibitem{BACHOCVALLENTIN}
Christine Bachoc and Frank Vallentin.
\newblock New upper bounds for kissing numbers from semidefinite programming.
\newblock {\em J. Am. Math. Soc.}, 21(3):909--924, 2008.

\bibitem{BANNAIBANNAI}
Eiichi Bannai and Etsuko Bannai.
\newblock A survey on spherical designs and algebraic combinatorics on spheres.
\newblock {\em Eur. J. Comb.}, 30(6):1392--1425, 2009.

\bibitem{BENEDETTOFICKUS}
John~J. Benedetto and Matthew Fickus.
\newblock Finite normalized tight frames.
\newblock {\em Adv. Comput. Math.}, 18(2-4):357--385, 2003.

\bibitem{BOURGAINMILMANWOLFSON}
J.~Bourgain, V.~Milman, and H.~Wolfson.
\newblock On type of metric spaces.
\newblock {\em Trans. Am. Math. Soc.}, 294:295--317, 1986.

\bibitem{KORNELSON}
J.~A. Chavez-Dominguez, D.~Freeman, and K.~Kornelson.
\newblock Frame potential for finite-dimensional {B}anach spaces.
\newblock {\em Linear Algebra Appl.}, 578:1--26, 2019.

\bibitem{COHNJIAOKUMARTORQUATO}
Henry Cohn, Yang Jiao, Abhinav Kumar, and Salvatore Torquato.
\newblock Rigidity of spherical codes.
\newblock {\em Geom. Topol.}, 15(4):2235--2273, 2011.

\bibitem{COHNKUMARMINTON}
Henry Cohn, Abhinav Kumar, and Gregory Minton.
\newblock Optimal simplices and codes in projective spaces.
\newblock {\em Geom. Topol.}, 20(3):1289--1357, 2016.

\bibitem{COHNZHAO}
Henry Cohn and Yufei Zhao.
\newblock Sphere packing bounds via spherical codes.
\newblock {\em Duke Math. J.}, 163(10):1965--2002, 2014.

\bibitem{CONWAYSLOANE}
J.~H. Conway and N.~J.~A. Sloane.
\newblock {\em Sphere packings, lattices and groups}, volume 290.
\newblock New York, NY: Springer, 3rd ed. edition, 1999.

\bibitem{DELSARTEGOETHALSSEIDEL}
P.~Delsarte, J.~M. Goethals, and J.~J. Seidel.
\newblock Spherical codes and designs.
\newblock {\em Geom. Dedicata}, 6:363--388, 1977.

\bibitem{ERICSONZINOVIEV}
Thomas Ericson and Victor Zinoviev.
\newblock {\em Codes on {Euclidean} spheres}, volume~63 of {\em North-Holland
  Math. Libr.}
\newblock North-Holland/Elsevier, 2001.

\bibitem{FARMERJOHNSON}
Jeffrey~D. Farmer and William~B. Johnson.
\newblock Lipschitz {{\(p\)}}-summing operators.
\newblock {\em Proc. Am. Math. Soc.}, 137(9):2989--2995, 2009.

\bibitem{JENSSENJOOSPERKINS}
Matthew Jenssen, Felix Joos, and Will Perkins.
\newblock On kissing numbers and spherical codes in high dimensions.
\newblock {\em Adv. Math.}, 335:307--321, 2018.

\bibitem{KRISHNAJOHNSON}
K.~Mahesh Krishna and P.~Sam Johnson.
\newblock Frames for metric spaces.
\newblock {\em Result. Math.}, 77(1):30, 2022.
\newblock Id/No 49.

\bibitem{LEMMENSSEIDEL}
P.~W.~H. Lemmens and J.~J. Seidel.
\newblock Equiangular lines.
\newblock {\em J. Algebra}, 24:494--512, 1973.

\bibitem{MENDELNAOR}
Manor Mendel and Assaf Naor.
\newblock Metric cotype.
\newblock {\em Ann. Math. (2)}, 168(1):247--298, 2008.

\bibitem{MUSIN}
Oleg~R. Musin.
\newblock The kissing number in four dimensions.
\newblock {\em Ann. Math. (2)}, 168(1):1--32, 2008.

\bibitem{ODLYZKOSLOANE}
A.~M. Odlyzko and N.~J.~A. Sloane.
\newblock New bounds on the number of unit spheres that can touch a unit sphere
  in n dimensions.
\newblock {\em J. Comb. Theory, Ser. A}, 26:210--214, 1979.

\bibitem{PFENDER}
Florian Pfender.
\newblock Improved {Delsarte} bounds for spherical codes in small dimensions.
\newblock {\em J. Comb. Theory, Ser. A}, 114(6):1133--1147, 2007.

\bibitem{PFENDERZIEGLER}
Florian Pfender and G{\"u}nter~M. Ziegler.
\newblock Kissing numbers, sphere packings, and some unexpected proofs.
\newblock {\em Notices Am. Math. Soc.}, 51(8):873--883, 2004.

\bibitem{POLYA}
G.~Polya.
\newblock {\em How to solve it: {A} new aspect of mathematical method}.
\newblock Princeton Sci. Libr. Princeton, NJ: Princeton University Press, 2014.

\bibitem{SAFFKUIJLAARS}
E.~B. Saff and A.~B.~J. Kuijlaars.
\newblock Distributing many points on a sphere.
\newblock {\em Math. Intell.}, 19(1):5--11, 1997.

\bibitem{SCHUTTEVANDER}
Kurt Schutte and B.~L. van~der Waerden.
\newblock Das {Problem} der dreizehn {Kugeln}.
\newblock {\em Math. Ann.}, 125:325--334, 1953.

\bibitem{SMALE}
Steve Smale.
\newblock Mathematical problems for the next century.
\newblock {\em Math. Intell.}, 20(2):7--15, 1998.

\bibitem{SARDARIZARGAR}
Naser Talebizadeh~Sardari and Masoud Zargar.
\newblock New upper bounds for spherical codes and packings.
\newblock {\em Math. Ann.}, 389(4):3653--3703, 2024.

\bibitem{VIAZOVSKA}
Maryna Viazovska.
\newblock Sharp sphere packings.
\newblock In {\em Proceedings of the international congress of mathematicians
  2018, Volume II.}, pages 455--466. Hackensack, NJ: World Scientific; Rio de
  Janeiro, 2018.

\bibitem{WALDRON}
Shayne F.~D. Waldron.
\newblock {\em An introduction to finite tight frames}.
\newblock Appl. Numer. Harmon. Anal. New York, NY: Birkh{\"a}user, 2018.

\bibitem{WOJTASZCZYK}
P.~Wojtaszczyk.
\newblock {\em Banach spaces for analysts}.
\newblock Cambridge: Cambridge Univ. Press, 1996.

\bibitem{ZONG}
Chuanming Zong.
\newblock {\em Sphere packings}.
\newblock Universitext. New York, NY: Springer, 1999.

\end{thebibliography}

\end{document}